\documentclass[12pt]{article}
\usepackage{amssymb,amsmath}
\usepackage{graphicx,color}
\usepackage[latin1]{inputenc}
\usepackage{mathrsfs}
\usepackage{array, amsfonts, mathrsfs}

\usepackage{amsmath, amsthm, amssymb, amstext}
\usepackage{array, amsfonts, mathrsfs}

\usepackage{latexsym, amsmath}

\newtheorem{Lemma}{Lemma}
\newtheorem{Theorem}{Theorem}

\date{}

\begin{document}

\title {Toeplitz matrices in the Boundary Control method}
\author{M.I.Belishev\thanks {St. Petersburg Department of Steklov Mathematical
        Institute, St.Petersburg, Russia, e-mail: belishev@pdmi.ras.ru. Supported
        by the RFBR grant 18-01-00269 and Volks-Wagen Foundation.},\,
         N.A.Karazeeva\thanks {St.Petersburg Department of Steklov Mathematical
        Institute, St. Petersburg, Russia, e-mail: karazeev@pdmi.ras.ru.}.}

\maketitle

\begin{abstract}
Solving inverse problems by dynamical variant of the BC-method is
basically reduced to inverting the connecting operator $C^T$ of
the dynamical system, for which the problem is stated. Realizing
the method numerically, one needs to invert the Gram matrix $\hat
C^T=\{(C^Tf_i,f_j)\}_{i,j=1}^N$ for a representative set of
controls $f_i$. To raise the accuracy of determination of the
solution, one has to increase the size $N$, which, especially in
the multidimensional case, leads to a rapid increase in the amount
of computations. However, there is a way to reduce it by the
proper choice of $f_j$, due to which the matrix $\hat C^T$ gets a
specific block-Toeplitz structure. In the paper, we explain, where
this property comes from, and outline a way to use it in numerical
implementation of the BC-algorithms.
\end{abstract}

\maketitle

\section{Introduction}

\noindent$\bullet$\,\,\, The BC-method  (boundary control method)
is an approach to inverse problems based on their connections with
control theory \cite{B Obzor IP 97}--\cite{B_UMN}. Its {\it local}
variant, which is considered in the paper, directs to application
in geophysics (seismology), where parameters of the medium
depending on the depth must be determined from the data on the day
surface. The wave process is initiated by the sources (controls)
disposed on a {\it part} of the boundary of the sounded domain;
the data are recorded on the same part. In accordance with the
finiteness of the wave propagation speed, the parameters must be
determined in the {\it real time}: the longer is the observation
time, the bigger is the depth of determination \cite{B Obzor IP
97,B How to see waves,CUBO_2,B EACM}.

Numerical algorithms based on the local variant are successfully
tested in the series of experiments \cite{B Obzor IP 97,BGotIv
COCV,BGotlib JIIPP,BIvKubSem,De Hoop,BIvKubSem
Marmousi,Oks,Pest_2010,Tim,YY}. Their further development and
possible use in the work with real data is an important
perspective goal. Our paper is adjacent to articles
\cite{BKaraz_2D} and \cite{BKaraz_3D}, which offer the simple
tests for probing algorithms based on the BC-method. It reveals
additional opportunities in numerical implementation of the
BC-method.
\smallskip

\noindent$\bullet$\,\,\,Typically, the inverse problems are {\it
nonlinear}. The principal advantage of the classical approach by
I.M.Gelfand, B.M.Levitan, M.G.Krein and V.A.Mar\-chenko is that it
reduces them to solving the {\it linear} problems and equations.
The BC-method inherits this advantage. It reveals a unified view
at the classical GLKM equations and shows that to solve them is to
invert the so-called {\it connecting operator $C^T$} of the
dynamical system associated with the forward problem (see
\cite{BMikh_JIIPP}). It is the fact, which  provides the relevant
multidimensional generalizations.

In accordance with these generalizations, to solve an inverse
problem by dynamical variant of the BC-method is basically to
invert the connecting operator $C^T$ of the relevant dynamical
system. Respectively, realizing the method numerically, one needs
to invert the Gram matrix $\hat C^T=\{(C^Tf_i,f_j)\}_{i,j=1}^N$
for a representative set of controls $f_i$, which simulates a
basis in the space of controls. Along this way, we encounter two
traditional obstacles:

\noindent 1) to raise the accuracy of determination of the
solution, one has to increase the number of controls (the size
$N$), which, especially in the multidimensional case, leads to a
rapid growth in the amount of computations;

\noindent 2) as the size grows, the matrix $\hat C^T$ becomes more
and more ill-conditional: the lower bound of its spectrum rapidly
tends to zero.

\noindent The latter is unavoidable: it reflects the strong
ill-posedness of {\it multidimensional} inverse problems, as well
as of the corresponding boundary control problem, which is being
solved within the BC-method. To deal with ill-posedness, a variety
of regularization methods are used (see the works cited above). We
pin certain hopes on A.A.Timonov's approach \cite{Tim}, which uses
the mean curvature flow technique.
\smallskip

\noindent$\bullet$\,\,\,The given paper deals with the difficulty
1). Its main subject is the remarkable fact that the matrix $\hat
C^T$ may have a specific {\it block-Toeplitz structure}. Due to
this, it is possible to significantly reduce the amount of
computations during the determination of the inverse matrix $[\hat
C^T]^{-1}$. We explain, where this property comes from, and
outline a way to use it in numerical implementation of the
BC-algorithms.

For the sake of determinacy and simplicity, we deal with a
concrete inverse problem treated in \cite{CUBO_2} and
\cite{BKaraz_2D,BKaraz_3D}. Namely, the simplicity is that the
geometry of rays in the domain, which the problem is considered
in, is Euclidean. However, the Toeplitz structure of $\hat C^T$ is
the fact of rather general character.
\smallskip

\noindent$\bullet$\,\,\,The authors thank I.V.Kubyshkin for the
kind help in computer graphics.

\section{Toeplitz structure}

\subsubsection*{Forward problems}
\noindent$\bullet$\,\, The initial boundary value problem
 \begin{align}
\label{F1} &u_{tt}-\Delta u-\langle\nabla \ln \rho, \nabla u\rangle = 0 &&\text{in} \quad {\mathbb R^2_+} \times (0,T), \\
\label{F2} &u|_{t=0} = u_t|_{t=0}=0 && \text{in} \quad \overline {\mathbb R^2_+},\\
\label{F3} &u_y|_{y=0}=f && \text{for} \quad 0\leqslant t\leqslant
T,
 \end{align}
is considered in the half-plain ${\mathbb
R^2_+}:=\{(x,y)\in{\mathbb R^2}\,|\,\,y>0\}$, where  $\Delta
u=u_{xx}+u_{yy}$,\, $\langle\nabla \ln \rho,\nabla u\rangle=(\ln
\rho)_xu_x+(\ln \rho)_yu_y$,\, $\rho=\rho(x,y)$ is a smooth
positive function ({\it reduced sound velocity}), $f=f(x,t)$ is a
Neumann {\it boundary control}, $T>0$ is a final time,
\,$u=u^f(x,y,t)$ is a solution ({\it wave}).
\smallskip

\noindent$\bullet$\,\,The solutions to the hyperbolic problem
(\ref{F1})--(\ref{F3}) obey a finiteness of the domain of
influence principle. Let $\sigma=\{(x,0)\,|\,\alpha\leqslant x
\leqslant \beta\}$ be a finite segment of the boundary $\mathbb
R_x:=\partial\mathbb R^2_+$,
 $$\Omega^r_\sigma:=\{p\in {\mathbb
R^2_+}\,|\,{\rm dist\,}(p,\sigma)<r\}
 $$
its metric neighborhood of radius $r>0$ ($\rm dist$ is the
Euclidean distance in $\overline{\mathbb R^2_+}$). If ${\rm
supp\,}f\subset \overline{\sigma}\times [0,T]$, i.e., the control
is supported on the segment $\sigma$, then
 \begin{equation}\label{Eq supp u^f}
{\rm supp\,}u^f(\,\cdot\,,\,\cdot\,, t)\,\subset\,\Omega^t_\sigma,
\quad 0<t\leqslant T,
 \end{equation}
holds, i.e., the corresponding wave is localized in the
$t$-neighborhood of the segment, from which the control acts.
\smallskip

\noindent Relation (\ref{Eq supp u^f}) means that the waves
propagate in a half-plane with the unit velocity.
\smallskip

\noindent$\bullet$\,\,By hyperbolicity, for the controls $f$
provided ${\rm supp\,}f\subset \overline{\sigma}\times [0,2T]$,
the following {\it extended} problem is also well posed:
 \begin{align}
\label{Ext1} &u_{tt}-\Delta u-\langle\nabla \ln \rho, \nabla u\rangle = 0, && p\in\Omega^T_\sigma,\,\,0<t<2T-{\rm dist\,}(p,\sigma); \\
\label{Ext2} &u=0\,&& \text{for} \quad t<{\rm dist\,}(p,\sigma);\\
\label{Ext3} &u_y|_{y=0}=f && \text{for} \quad 0\leqslant
t\leqslant 2T,
 \end{align}
where $p=(x,y)\in {\mathbb R^2_+}$.

\subsubsection*{Dynamical  system}
Problem (\ref{F1})--(\ref{F3}) is endowed with standard attributes
of a dynamical system: spaces and operators.
\smallskip

\noindent$\bullet$\,\,An {\it outer space} is the space ${\mathscr
F}^T_\sigma:=L_2(\sigma\times[0,T])$ of controls acting from
$\sigma$, with the inner product
 \begin{equation}\label{Eq product in F}
(f,g)_{{\mathscr F}^T_\sigma} = \int\limits_{\sigma \times [0,T]}
f(x,t)\,g(x,t)\,\rho(x,0) \,dx\,dt.
 \end{equation}

\noindent$\bullet$\,\,An {\it inner space} of the states (waves)
is ${\mathscr H^T_\sigma}:=L_{2,\rho}({\Omega^T_\sigma})$ with the
inner product
 \begin{equation*}
(v,w)_{{\mathscr H^T_\sigma}}= \int\limits_{{\Omega^T_\sigma}}
v(x,y)\,w(x,y)\,\rho(x,y)\,dx\,dy.
 \end{equation*}
By virtue of (\ref{Eq supp u^f}) the waves
$u^f(\,\cdot\,,\,\cdot\,,t)$ are its elements (the functions of
the variables $x$ and $y$ depending on $t$ as a parameter).
\smallskip

\noindent$\bullet$\,\,A {\it control operator} $W^T: {\mathscr
F}^T _\sigma\to {\mathscr H}^T_\sigma$,
 \begin{equation}\label{Eq W^T}
(W^Tf)(x,y):=u^f(x,y,T), \quad (x,y)\in \mathbb R^2_+
 \end{equation}
is the operator, which solves problem (\ref{F1})--(\ref{F3}). It
is compact and injective for any  $T>0$ \cite{CUBO_2}.
\smallskip

\noindent$\bullet$\,\,A {\it response operator} $R^T: {\mathscr
F}^T_\sigma \to {\mathscr F}^T_\sigma$,
 \begin{equation*}
(R^Tf)(x,t):=u^f(x,0,t), \qquad x\in\sigma,\enskip 0\leqslant t
\leqslant T
 \end{equation*}
describes the reaction of the system on the effect of control. It
is also compact. The right-hand side in this definition means a
pressure observed on the boundary of the half-plane. The
well-known representation
 \begin{equation*}
(R^Tf)(x,t):=\int\limits_0^t ds\int\limits_\sigma
r(x,x';t-s)\,f(x',s)\,dx', \qquad x\in\sigma,\enskip 0\leqslant t
\leqslant T
 \end{equation*}
holds with a piece-wise continuous kernel $r$.
\smallskip

\noindent$\bullet$\,\,An {\it extended} response operator $R^{2T}:
{\mathscr F}^{2T}_\sigma \to {\mathscr F}^{2T}_\sigma$,
 \begin{equation*}\label{Eq R^T ext def}
(R^{2T}f)(x,t):=u^f(x,0,t), \qquad x\in\sigma,\enskip 0\leqslant t
\leqslant 2T
 \end{equation*}
is associated with the problem (\ref{Ext1})--(\ref{Ext3}) and the
representation
 \begin{equation}\label{Eq R^T ext repres}
(R^{2T}f)(x,t):=\int\limits_0^t ds\int\limits_\sigma
r(x,x';t-s)\,f(x',s)\,dx', \qquad x\in\sigma,\enskip 0\leqslant t
\leqslant 2T
 \end{equation}
holds. Both operators $R^{T}$ and $R^{2T}$ are compact. Also, both
of them are determined by the values of the function $\rho$ in
$\Omega^T_\sigma$ only (do not depend on $\rho$ outside
$\Omega^T_\sigma$).
\smallskip

\noindent$\bullet$\,\,An operator $C^T:{\mathscr F}^T_\sigma \to
{\mathscr F}^T_\sigma$,
 \begin{equation}\label{Eq C^T}
C^T:=(W^T)^\ast W^T
 \end{equation}
is called {\it connecting operator}. For controls $f,g \in
{\mathscr F}^T_\sigma$ one has
\begin{equation*}
\left(u^f(\,\cdot\,,\,\cdot\,,T),u^g(\,\cdot\,,\,\cdot\,,T)\right)_{{\mathscr
H}^T_\sigma}\overset{(\ref{Eq W^T})}=(W^Tf,W^Tg)_{{\mathscr
H}^T_\sigma}=(C^Tf,g)_{{\mathscr F}^T_\sigma};
\end{equation*}
so  $C^T$ connects the metrics of the outer and inner spaces. In
view of compactness and injectivity of  $W^T$, the operator $C^T$
is also compact and injective. Also, by its definition (\ref{Eq
C^T}), $C^T$ is self-adjoint and positive.
\smallskip

\noindent$\bullet$\,\,\,One of the key facts of BC-method is a
simple and explicit relation between the response and connecting
operators (see \cite{B Obzor IP 97,B How to see waves,CUBO_2}).
For its formulation we introduce:

the operator of the odd extension $S^T: {\mathscr F}^T_\sigma \to
{\mathscr F}^{2T}_\sigma$,
 $$
(S^T f)(\,\cdot\,,t)\,:=\, \begin{cases}
                       f(\,\cdot\,,t), &0\leqslant t<T;\\
                       - f(\,\cdot\,,2T-t), &T\leqslant t \leqslant 2T;
                       \end{cases}
 $$

the integration $J^{2T}:{\mathscr F}^{2T}_\sigma \to {\mathscr
F}^{2T}_\sigma$,
 $$
(J^{2T}f)(\,\cdot\,,t):=\int\limits_0^t f(\,\cdot\,,s)\,ds\,,
\qquad 0\leqslant t \leqslant 2T;
 $$

the operator $P^{2T}: {\mathscr F}^{2T}_\sigma\to{\mathscr
F}^{2T}_\sigma$ that selects the odd part of `long' controls:
 $$
(P^{2T} f)(\,\cdot\,,t)\,:=\,{1\over 2}\,\left[f(\,\cdot\,,t)-
f(\,\cdot\,,2T-t)\right]\,, \qquad 0\leqslant t \leqslant 2T;
 $$

the reduction operator $N^{2T}: {\mathscr
F}^{2T}_\sigma\to{\mathscr F}^{T}_\sigma$,
 $$
(N^{2T}f)(\,\cdot\,,t):= f(\,\cdot\,,t)\,, \qquad 0\leqslant t
\leqslant T.
 $$
Also, note a simply verified relation $(S^T)^*=2N^{2T}P^{2T}$.
\begin{Lemma}
The representation
 \begin{equation}\label{Eq C^T via R^2T}
C^T\,=\,\frac{1}{2}\,(S^T)^\ast J^{2T} R^{2T} S^T
 \end{equation}
is valid.
\end{Lemma}

\subsubsection*{Inverse problem}

In the domain $\Omega^T_\sigma$ filled with waves at the final
moment $t=T$, one selects a subdomain ({\it ray tube})
 $$
B^T_\sigma\,:=\,\{(x,y)\,|\,\,x\in\sigma,\,\,0\leqslant y<T\}\,,
 $$
which is covered by the rays emanated from the points of $\sigma$
orthogonally to the boundary (see Fig.1).

The {\it inverse problem} is to determine the function $\rho$ in
the tube  $B^T_\sigma$ from the given extended response operator
$R^{2T}$. Also, for the sake of simplicity, we assume that the
function $\rho(\cdot,0)$ is known in $\sigma$ that enables one to
use the product (\ref{Eq product in F}).

\begin{figure}[htp]
\begin{center}
\includegraphics[width=4in]{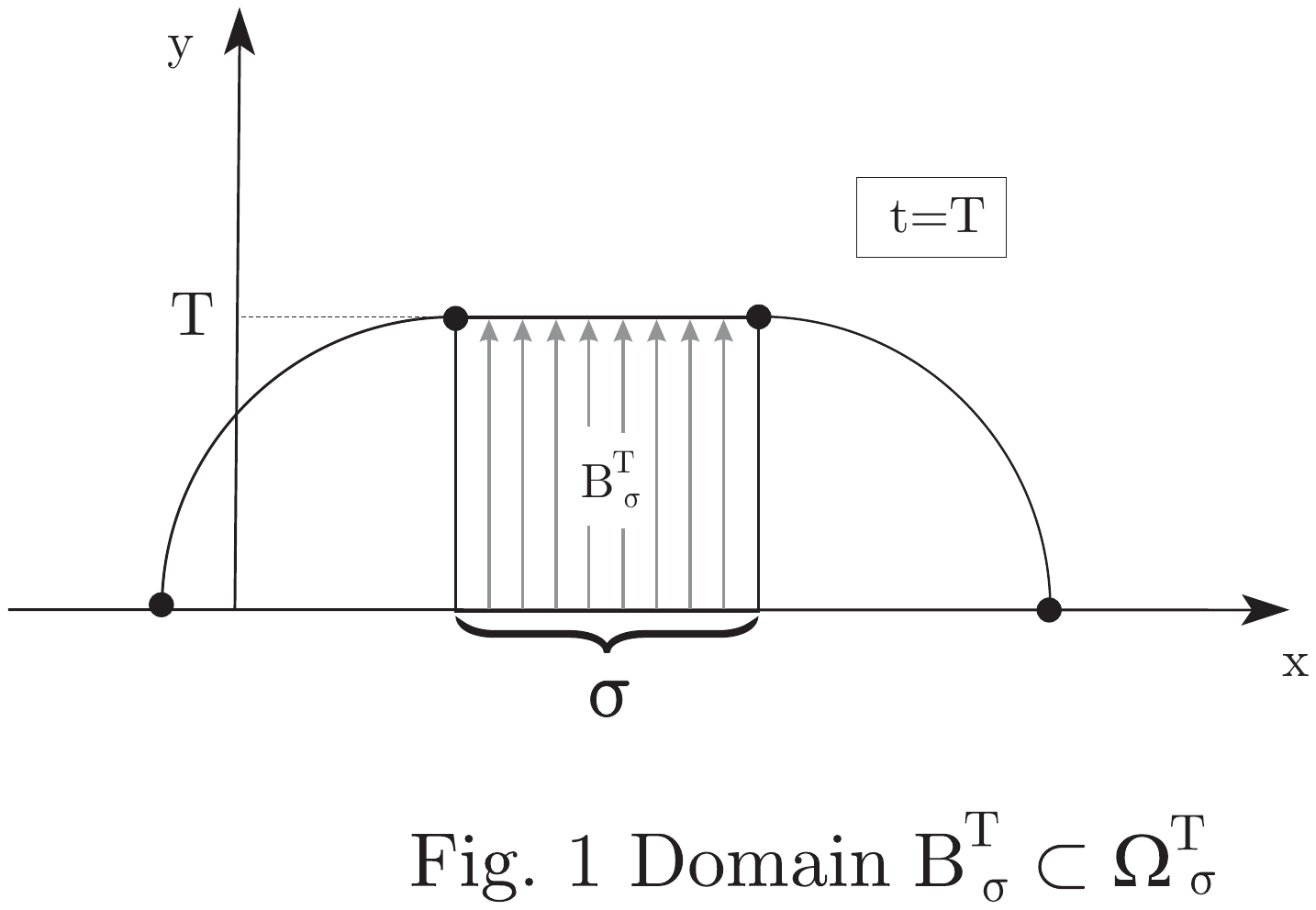}
\end{center}
\end{figure}

\subsubsection*{BCP}\label{BCP}

\noindent$\bullet$\,\,As usual in the BC-method, one relates the
inverse problem with a corresponding boundary control problem
(BCP). For the problem stated above, the relevant BCP is the
following.

Let $1^T_\sigma\in{\mathscr H^T_\sigma}$ be the function equal to
$1$ identically in ${\Omega^T_\sigma}$. The BCP is to find the
control $f\in{\mathscr F^T_\sigma}$, which provides
 \begin{equation}\label{Eq BCP}
u^f(\cdot,\cdot,T)\,=\,1^T_\sigma.
 \end{equation}
To be more precise, for determination of $\rho|_{B^T_\sigma}$ one
needs to solve a family of the `shortened' problems (\ref{Eq BCP})
with the r.h.s. $1^\xi_\sigma\in\mathscr H^\xi_\sigma$ for all
$0<\xi\leqslant T$ and then apply the so-called {\it amplitude
formula} \cite{CUBO_2,BKaraz_2D}. However, to demonstrate the
appearance of Toeplitz matrices, it suffices to consider a single
problem (\ref{Eq BCP}).
\smallskip

\noindent$\bullet$\,\,Writing (\ref{Eq BCP}) in the form
$W^Tf=1^T_\sigma$ (see (\ref{Eq W^T})) and applying the adjoint
operator, we get $(W^T)^*W^Tf=(W^T)^*1^T_\sigma$. Integration by
parts provides $(W^T)^*1^T_\sigma=\varkappa^T$, where
$\varkappa^T\in{\mathscr F^T_\sigma}:\varkappa^T(x,t):=T-t$ (see
\cite{CUBO_2}). At last, recalling the definition of the
connecting operator, we arrive at the equation
 \begin{equation}\label{Eq C^Tf=kappa}
C^T f\,=\,\varkappa^T \qquad {\rm in}\,\,\,{\mathscr
F^T_\sigma}\,.
 \end{equation}
For the subsequent, it is convenient to reduce (\ref{Eq
C^Tf=kappa}) to an equation in $\mathscr
F^{2T}_\sigma=L_2([0,2T]\times\sigma)$ as follows.

Denote $\tilde f:=S^Tf\in\mathscr F^{2T}_\sigma$ and $p(x,x'; t)
:= \int\limits _0^t r(x, x'; s)\,ds$..
 \begin{Lemma}
Equation (\ref{Eq C^Tf=kappa}) is equivalent to the equation in
$\mathscr F^{2T}_\sigma$ of the form
 \begin{equation}\label{Eq Basic Eq}
{\int \limits _0^{2T} d s \int \limits _\sigma p(x,x'; |t -
s|)\,\tilde
f(x',s)\,dx'=4\,(T-t),\quad(x,t)\in\sigma\times[0,2T].}
 \end{equation}
 \end{Lemma}
\begin{proof} By (\ref{Eq R^T ext repres}), the operator
$M^{2T}:=J^{2T}R^{2T}$ acts in $\mathscr F^{2T}_\sigma$ by the
rule
 \begin{align*}
& \left ( M^{2T} f \right) (x,t)= \int \limits _0^t d\, s \int
\limits _\sigma p(x, x'; t -s)\,f(x',s)\,dx'=\\
&=\int \limits _0^t d\, s \int \limits _\sigma p(x, x'; |t
-s|)\,f(x',s)\,dx' ,\quad (x,t)\in\sigma\times[0,2T].
 \end{align*}
 Its
adjoint acts by
 \begin{align*}
& \left((M^{2T})^* f \right)(x,t) = \int \limits _t^{2T} d s \int
\limits _\sigma p(x,x'; s-t) \,f(x',s)\,d x'=\\
&=\int \limits _t^{2T} d s \int \limits _\sigma p(x,x'; |s-t|)
\,f(x',s)\,d x',\quad (x,t)\in\sigma\times[0,2T].
 \end{align*}
As a result, we have
 \begin{align}
\notag &  \left (\left[ M^{2T} + (M^{2T})^* \right]f\right)
(x,t)=\\ \label{Eq M+M*}& = {\int \limits _0^{2T}} d s \int
\limits _\sigma p(x,x';{|t - s|}) \,f(x',s)\, d x',\quad
(x,t)\in\sigma\times[0,2T].
 \end{align}
In the meantime, (\ref{Eq C^T via R^2T}) implies
 \begin{align*}
& C^T = (C^T)^* =
\frac{1}{2}\,\left[C^T+(C^T)^*\right]=\\
&=\frac{1}{2}\,\left[\frac{1}{2}\,(S^T)^*J^{2T}R^{2T}S^T+\frac{1}{2}\,(S^T)^*(J^{2T}R^{2T})^*S^T\right]=\\
& = \frac {1}{4}\,(S^T)^* \left[ M^{2T} + (M^{2T})^* \right]S^T\,.
 \end{align*}

Applying $S^T$ to (\ref{Eq C^Tf=kappa}) and using $S^T(S^T)^*=I$,
we arrive at
 $$
\left[ M^{2T} + (M^{2T})^* \right]\tilde
f=4\,{\tilde\varkappa}^{T}\qquad \text{in}\,\,\mathscr
F^{2T}_\sigma,
 $$
where $\tilde
f:=S^Tf,\,\,{\tilde\varkappa}^{T}:=S^T\varkappa^T=T-t$. Then
(\ref{Eq M+M*}) implies (\ref{Eq Basic Eq}).
 \end{proof}
To get the solution $f$ to (\ref{Eq C^Tf=kappa}), one can solve
(\ref{Eq Basic Eq}) and then take $f=\tilde f|_{0\leqslant
t\leqslant T}$.
\smallskip

Equations (\ref{Eq C^Tf=kappa}) and (\ref{Eq Basic Eq}) are the
relevant multidimensional analogs of the classical GLKM equations:
see \cite{B_mult_GLKM,BMikh_JIIPP}.

\subsubsection*{Toeplitz matrix}

\noindent$\bullet$\,\,Solving (\ref{Eq Basic Eq}) numerically, the
solution $\tilde f$ is sought in the form of an expansion over a
complete linearly independent system of controls in $\mathscr
F^{2T}_\sigma$. Such a system is simulated by a finite
sufficiently rich system of controls, to the description of which
we proceed.
\smallskip

In the space-time plane $\mathbb R^2_{(x,t)}$, choose a function
({\it basic source})
 $$
g=g(x,t):\quad{\rm
supp\,}g=\Delta_0^0:=[0,\varepsilon]\times[0,\delta]\qquad(\varepsilon,\delta>0)
 $$
supported on the {\it basic rectangular} $\Delta_0^0$; let
$$\Delta^i_j:=\{(x+j\varepsilon,t+i\delta)\,|\,\,(x,t)\in\Delta^0_0,\,\,\,j=0,1,\dots,M-1;\,\,i=0,1,\dots,N-1\}$$
(see Fig.2) and
$g^i_j(x,t):=g(x-j\varepsilon,t-i\delta),\,\,g^0_0:=g$. For the
fixed (big) $M$ and $N$, fit $\varepsilon=|\sigma|\slash
M,\delta=(2T)\slash N$ to provide the partition
 $$
\sigma\times[0,2T]=\bigcup\limits_{\substack{j=0,1,...,M-1;\\i=0,1,...,N-1}}\Delta^i_j\,.
 $$
Take {$g^i_j=g^i_j(x,t):=g(x-j\varepsilon,t-i\delta)$} supported
in {$\Delta_j^i$}. By their choice, the controls $g^i_j$ are
linearly independent.

\begin{figure}[htp]
\begin{center}
\includegraphics[width=4in]{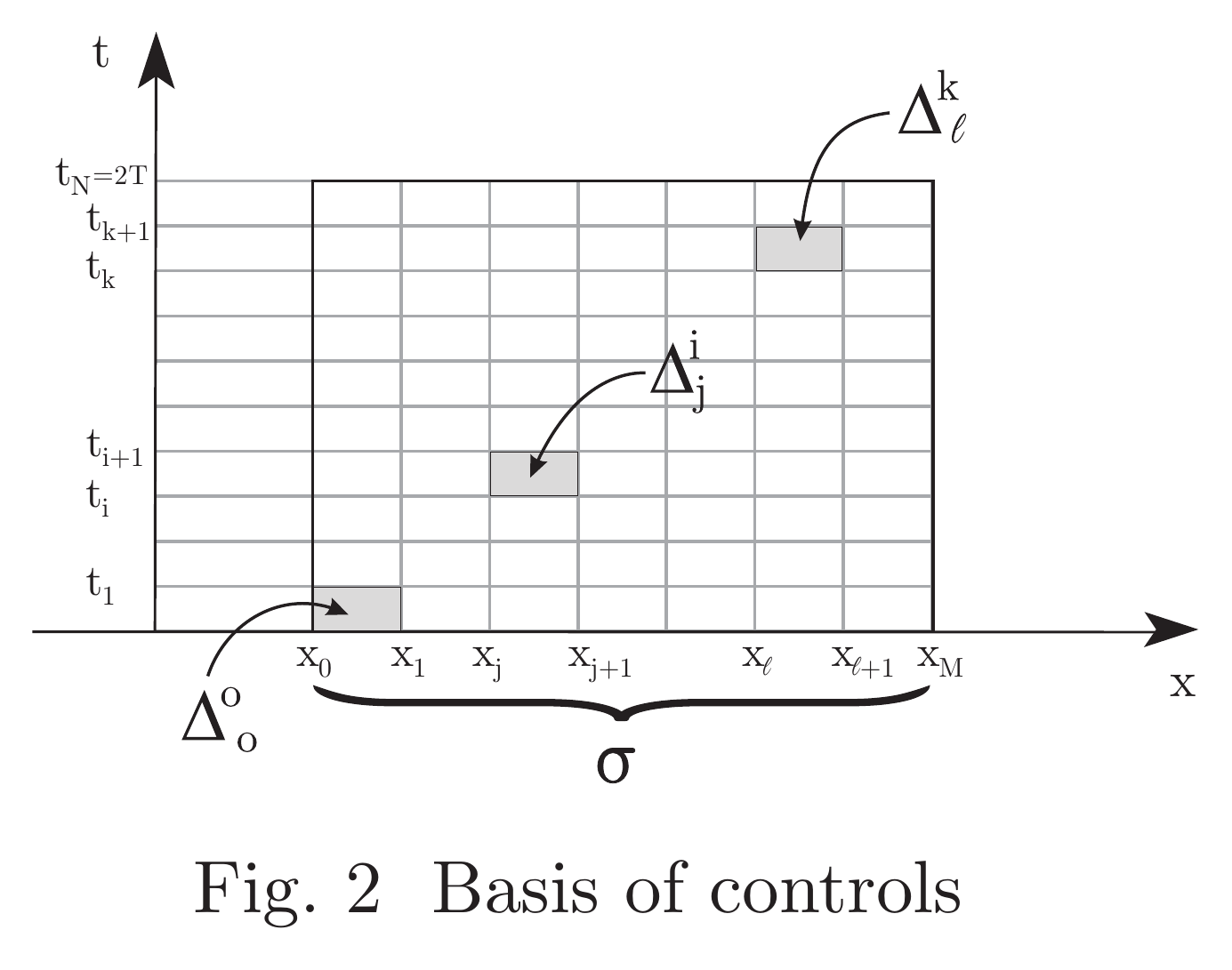}
\end{center}
\end{figure}

\noindent$\bullet$\,\,Solving (\ref{Eq Basic Eq}), we seek for the
solution in the form
 $$
\tilde
f(x,t)=\sum\limits_{\substack{j=0,1,...,M-1;\\i=0,1,...,N-1}}c^i_j\,g_j^i(x,t)
 $$
with the unknown $c^i_j$ and arrive at the linear system
 \begin{equation}\label{Eq Gram system}
\sum\limits_{\substack{j=0,1,...,M-1;\\i=0,1,...,N-1}}G^{ik}_{jl}\,c^i_j=4({\tilde\varkappa}^{T},g^k_l)_{\mathscr
F^{2T}_\sigma},\qquad j,l=1,\dots,M-1;\,\,i,k=1,\dots,N-1\quad
 \end{equation}
with the Gram matrix
\begin{align*}
&
G^{ik}_{jl}:=\left([M^{2T}+(M^{2T})^*]g^i_j\,,\,g^k_l\right)_{\mathscr
F^{2T}}\overset{(\ref{Eq Basic Eq})}=\\
&=\int \limits _{\Delta^k_l}dx\,dt\,\, g^k_l(x,t) \int \limits
_{\Delta^i_j} p(x,x'; |t - s|)\,\tilde g^i_j(x',s)\,ds\,dx'=\\
&=\int \limits _{\Delta^k_l}dx\,dt\,\,
g(x-l{\varepsilon},t-k\delta) \int \limits _{\Delta^i_j} p(x,x';
|t - s|)\,\tilde g(x'-{\varepsilon}
j,s-i\delta)\,ds\,dx'=\\
&=\int \limits _{\Delta^0_0}dx\,dt\,\, g(x,t) \int \limits
_{\Delta^0_0} p(x+l{\varepsilon},x'+{\varepsilon}
j; |t-k\delta - (s-i\delta)|)\,\tilde g(x',s)\,ds\,dx'=\\
&=\int \limits _{\Delta^0_0}dx\,dt\,\, g(x,t) \int \limits
_{\Delta^0_0} p(x+l{\varepsilon},x'+{\varepsilon} j; |t-s
+{(k-i)}\delta|)\,\tilde g(x',s)\,ds\,dx.
\end{align*}
Its entries depend on the difference $k-i$, so that we have a {\it
block-T$\ddot{\text{o}}$plitz matrix} with $N\times N$ blocks of
the size $M\times M$. Such a peculiarity of the Gram system is due
to the proper choice of the controls $g^i_j$. For the first time,
it was used in numerical testing by V.Yu.Gotlib in \cite{BGotIv
COCV,BGotlib JIIPP}.

The r.h.s. of the system (\ref{Eq Gram system}) is
 \begin{equation}\label{Eq beta}
4\,({\tilde\varkappa}^{T},g^k_l)_{\mathscr
F^{2T}_\sigma}=4\int\limits_{\sigma\times[0,2T]}(T-t)\,g^k_l(x,t)\,dx\,dt=4\int\limits
_{\Delta^k_l}(T-t)\,g^k_l(x,t)\,dx\,dt=:\beta^k_l.
 \end{equation}
Thus, the system is fully determined by the inverse data $R^{2T}$
and the choice of controls $g^i_j$.

\section{Inverting algorithm}

The algorithms provided in \cite{VT} are apropriate for a wide
class of nonsingular block-Toeplitz matrices. However, the Gram
matrix $G^{ik}_{jl}$, with which we deal, is not only
block-Toeplitz but also {\it symmetric and positive}. This is the
case we are considering in this section.
\smallskip

\noindent$\bullet$\,\, So, we deal with a nonsingular symmetric
positive block-Toeplitz matrix
 $${\bf G} = \{G^{ik}_{jl}\}_{\substack{j,l=0,1,...,M-1;\\i,k=0,1,...,N-1}} =
 \begin{pmatrix}
\gamma_0& \gamma_{1}&\gamma_{2}&\ldots&\gamma_{N-1}\\
\gamma_1& \gamma_0&  \gamma_{1}& \ldots& \gamma_{N-2}\\
\gamma_2 & \gamma_1 & \gamma_0&\ldots & \gamma_{N-3}\\
\vdots& \vdots &\vdots&\ddots&\vdots\\
\gamma_{N-1}&\gamma_{N-2}&\gamma_{N -3}&\ldots&\gamma_0
\end{pmatrix},$$
where the blocks $\gamma_m$ are the $M\! \times\! M$ -\,matrices
with the entries $G_{jl}^{ik},$\\ $|i - k| = m;$ $m = 0,
\ldots,N-1.$ The indices $j$ and $l$ take the values $0 \leqslant
j,\,l \leqslant M -1$.

Introduce the row of the unknowns
$$ {\mathscr C} := \left( c_0^0, \ldots, c_{M-1}^0;c_0^1, \ldots,
c_{M-1}^1; \ldots; c_0^{N-1}, \ldots, c_{M-1}^{N-1} \right)$$ and
the row of the r.h.s. $\beta^k_l$
$$ {\mathscr B} := \left ( \beta_0^0, \ldots, \beta_{M-1}^0;
\beta_0^1, \ldots, \beta_{M-1}^1;\ldots; \beta_0^{N-1}, \ldots,
\beta_{M-1}^{N-1} \right)$$ defined by (\ref{Eq beta}), the rows
being of the length $MN$. Then (\ref{Eq Gram system}) is
equivalent to the system
\begin{equation}\label{Eq AG=B}
{\mathscr C} {\bf G} = {\mathscr B}.
\end{equation}
To solve (\ref{Eq AG=B}) one needs to invert the matrix ${\bf G}$.
For this, we apply a Levinson type algorithm presented in \cite
{VT}. The construction of the inverse matrix relies on the Theorem
\ref{Th 1} taken from \cite{VT} and exposed below.

In what follows, the accent $'$ denotes the operation of the block
transposition. That is, the blocks are transposed but their inner
structure is not disturbed.

 Let
$Y=(Y_0,\ldots,Y_{N-1})'$ be the block column, which satisfies the
relation
 \begin{equation} \label{GX-GY}
{\bf G} Y = \mathscr I_N = \begin{pmatrix} O\\O\\
\vdots \\I
\end{pmatrix},
 \end{equation}
where {$I$ and $O$} are the unit and zero $M\!\times\! M$
-\,matrices. Such  $Y$ does exist and is unique just because $\bf
G$ is nonsingular. The block column  $\mathscr I_N$ is the matrix
consisting of $MN$ rows and $M$ columns;  $Y$ is also a matrix of
the same size. Moreover, as is shown in \cite{VT}, the
$M\!\times\!M$- matrix $Y_{N-1}$ is necessarily {\it nonsingular}.
 \begin{Theorem}\label{Th 1}
Let $Y$  be determined by (\ref{GX-GY}). Then the representation
 \begin{align}
\notag &{\bf G}^{-1} = \begin{pmatrix}
Y_{N-1}&O&\ldots&O\\
Y_{N-2}&Y_{N-1}&\ldots&O\\ \vdots& \vdots&\ddots& \vdots\\
Y_0&Y_1&\ldots&Y_{N-1} \end{pmatrix} Y_{N-1}^{-1} \begin{pmatrix}
Y_{N-1}&Y_{N-2}&\ldots&Y_0\\ O&Y_{N-1}&\ldots&Y_1\\
\vdots&\vdots&\ddots&\vdots\\ O&O&\ldots&Y_{N-1}
\end{pmatrix} -\\
\label{Eq G^{-1}}&- \begin{pmatrix} O&O&\ldots&O\\ Y_0&O&\ldots&O\\
\vdots&\vdots&\ddots&\vdots\\Y_{N-2}&Y_{N-3}&\ldots&O
\end{pmatrix} Y_{N-1}^{-1} \begin{pmatrix}
O&Y_0&\ldots&Y_{N-2}\\ O&O&\ldots&Y_{N-3}\\ \vdots& \vdots&
\ddots&\vdots\\ O&O &\ldots&O \end{pmatrix}
 \end{align}
is valid.
\end{Theorem}
\smallskip

\noindent$\bullet$\,\,\,By positivity of $\bf G$, all the matrices
 $$ {\bf G}_k = \begin{pmatrix} \gamma_0& \gamma_1&\ldots& \gamma_k\\
\gamma_1&\gamma_0&\ldots& \gamma_{k-1}\\
\vdots& \vdots& \ddots& \vdots\\
\gamma_k& \gamma_{k-1} & \ldots& \gamma_0
\end{pmatrix}$$
of the order $k+1\,$ $(k = 0, \ldots, N-1)$ are symmetric and
positive definite. In particular, $\gamma_0$ is positive definite
and, hence, invertible.

To find the block column  $Y$ one can use the following recurrent
Procedure. Let
$$ Y^{(k)} = \left ( Y_0^{(k)}, Y_1^{(k)}, \ldots, Y_k^{(k)} \right )'$$
be found at the previous steps. So, $Y^{(k)}$ consists of $k+1$
blocks of the size $M \times M$ and satisfies the system
(\ref{GX-GY}) of the size $k+1$ ($0\leqslant k\leqslant N-1$).

The Levinson algorithm \cite{VT}, which we apply, makes the use of
the {\it normalizing factors}  $Q_k$ ($M \times M$-matrix)
entering in the representation
 $$
Y_l^{(k)} = \tilde Y_l^{(k)} Q_k; \qquad l = 0, \dots,
k\,\,\,\,\,(k = 0,1, \dots, N-1)
 $$
and being also determined during the Procedure.
\smallskip

\noindent$\bullet$\,\,\,The Procedure is as follows.
\begin{enumerate}
\item At the first step we determine $\tilde Y^{(0)}_0 =
\gamma_{0}^{-1} Q_0^{-1}$, where $Q_0$ is arbitrary nonsingular
$M\! \times\! M$\,-\,matrix. In particular one can take $Q_0 = I$.

\item Let the block vector $\tilde Y^{(k-1)}$ and factor $Q_{k-1}$
be already found. Introduce the auxiliary matrices $ E_k$ and
$F_k$ by
 $$
E_k = \gamma_1 \tilde Y_0^{(k-1)} + \gamma_2 \tilde Y_1^{(k-1)} +
\ldots + \gamma_k \tilde Y_{k-1}^{(k-1)}, \qquad F _k = - Q_{k-1}
E_k;
 $$
whereas the matrix $I - F_k^2$ turns out to be nonsingular \cite
{VT}. Then the factor $Q_k$ is determined by
$$ Q_k = (I - F_k^2) ^{-1} Q_{k-1}.$$

\item At the next step one computes the block vector $\tilde
Y^{(k)}$ by
\begin{align*}
& \left( \tilde Y^{(k)}_0, \tilde Y_1^{(k)}, \ldots, \tilde Y_k
^{(k)} \right)' =  \left( \tilde Y_{k-1}^{(k-1)}, \tilde
Y_{k-2}^{(k-1)},
\ldots, \tilde Y_0^{(k-1)} ,0 \right)' F_k +\\
&+ \left( 0, \tilde Y_0^{(k-1)}, \tilde Y_1^{(k-1)}, \ldots,
\tilde Y_{k-1}^{(k-1)} \right)'
\end{align*}
and determines
$$ Y_l^{(k)} =  \tilde Y_l^{(k)} Q_k ; \quad l= 0, \ldots, k. $$

\item Proceeding up to $k=N-1$, we find the block column $
Y^{(N-1)} = Y$. Then, by the use of (\ref{Eq G^{-1}}), we get the
matrix ${\bf G^{-1}}$.
\end{enumerate}
\smallskip

The algorithm requires $\sim  M^3 N^2$ operations and, thus, may
be regarded as "fast" \cite{VT}.
\smallskip

Inversion of Toeplitz matrices is also treated in the paper
\cite{HR}. However, it does not deal with the {\it block}
matrices, as is necessary in multidimensional problems. In the
meantime, perhaps, the algorithms of \cite {HR} may be developed
for this case.


\begin{thebibliography}{99}

\bibitem{B_mult_GLKM}
M.I.Belishev. {\it The Gelfand--Levitan type equations in
multidimensional inverse problem for the wave equation.} Zap.
Nauchn. Semin. LOMI, 165, 15-20 (Engl. transl. 1990 J. Sov. Math.,
1990, v.50, No 6, P. 1940-1944).

\bibitem{B Obzor IP 97}
M.I.Belishev. {\it Boundary control in reconstruction of manifolds
and metrics (the BC method).} Inverse Problems, {\bf13}, No.~5
(1997), R1--R45.

\bibitem{B How to see waves}
M.I.Belishev.
 {\it How to see waves under the Earth surface (the BC-method for
geophysicists)}.  Ill-Posed and Inverse Problems. S. I. Kabanikhin
and V. G. Romanov (Eds). VSP, 2002, pp.~55-72.

\bibitem{CUBO_2}
M.I.Belishev. {\it Dynamical Inverse Problem for the Equation
$u_{tt}-\Delta u - \nabla \rho \,\cdot\, \nabla u=0$ (the
BC-Method)}. {CUBO A Math. J.}  {\bf10},  No, 2 (2008), 17-33.

\bibitem{B EACM}
M.I.Belishev. {\it Boundary Control Method.} Encyclopedia of
Applied and Computational Mathematics, Vol. 1, pp.~142-146.

\bibitem{B_UMN}
M.I.Belishev. {\it Boundary control and tomography of Riemannian
manifolds (BC-method).} Uspechi Mat. Nauk  {\bf 72} (2017), No. 4,
3-66 (in Russian).

English translation: M.I.Belishev.
\newblock{Boundary control and tomography of Riemannian
manifolds (the BC-method).} Russian Math. Surveys; DOI
10.1070/RM9768.

\bibitem{BGotlib JIIPP}
M.I.Belishev, V.Yu.Gotlib.
 {\it Dynamical variant of the BC-method: theory and numerical testing.}
J. Inverse and Ill-Posed Problems {\bf 7}, No. 3 (1999), 221-240.

\bibitem{BGotIv COCV}
M.I.Belishev, V.Yu.Gotlib, S.A.Ivanov. {\it The BC-method in
multidimensional spectral inverse problem: theory and numerical
illustrations.} Control, Optimization and Calculus of Variations,
2 (1997), October, 307-327.

\bibitem{BIvKubSem}
M.I.Belishev, I. B. Ivanov, I. V. Kubyshkin, V. S. Semenov. {\it
Numerical testing in determination of sound speed from a part of
boundary by the BC-method.} J. Inverse and Ill-Posed Problems
{\bf24} (2016), Issue 2, 159--180.  DOI: 10.1515/jiip-2015-0052.

\bibitem{BKaraz_2D}
M.I.Belishev, N.A.Karazeeva. {\it Simplest Test for the
Two-Dimensional Dynamical Inverse Problem (BC-Method).} J Math
Sci, (2019). https://doi.org/10.1007/s10958-019-04567-5.

\bibitem{BKaraz_3D}
M.I.Belishev, A.S.Blagoveshchensky, N.A.Karazeeva. {\it Simplest
Test for the Three-Dimensional Dynamical Inverse Problem (The
BC-Method).} J Math Sci (2021).
https://doi.org/10.1007/s10958-021-05182-z.

\bibitem{BMikh_JIIPP}
M.I.Belishev, V.S.Mikhailov {\it Unified approach to classical
equations of inverse problem theory.} Journal of Inverse and
Ill-Posed Problems, 20 (2012), no 4, 461--488.

\bibitem{BIvKubSem Marmousi}
I.B.Ivanov, M.I.Belishev, V.S.Semenov. {\it The reconstruction of
sound speed in the Marmousi model by the boundary control method.}
{\tt arXiv: 1609.07586v1 [physics.geo-ph] 24 Sept 2016.}

\bibitem{HR}
G.Heinig, K.Rost. {\it Fast algorithms for Toeplitz and Hankel
matrices.} Linear Algebra and its Applications, 435 (2011, 1-59.
doi:10.1016/j.laa.2010.12.001.

\bibitem{De Hoop}
M.V.De Hoop, P.Kepley, L.Oksanen.
  {\it Recovery of a smooth metric via wave field and
coordinate transformation reconstruction.}  SIAM J. Appl. Math.
{\bf78}, No. 4 (2018), 1931-1953.

\bibitem{Oks}
L.Oksanen. {\it Solving an inverse obstacle problem for the wave
equation by using the boundary control method.} Inverse Problems
{\bf29} (2013), No. 3, 035004; doi:10.1088/0266-5611/29/3/035004.

\bibitem{Pest_2010}
L.Pestov, V.Bolgova and O.Kazarina. {\it Numerical recovering of a
density by the BC-method}. Inverse Problems and Imaging, Volume 4,
No. 4, 2010, 703-712. doi:10.3934/ipi.2010.4.703.

\bibitem{Tim}
A.A.Timonov. {\it A novel method for the numerical solution of a
hybrid inverse problem of electrical conductivity imaging.}
Zapiski Nauch. Semin. POMI, 499 (2020), 105-128\quad(in Russian).

\bibitem{VT}
V.V. Voevodin, E.E. Tyrtyshnikov. {\it Numerical calculus of
Toeplitz matrices }. Moscow, Nauka, 1987 (in Russian).

\bibitem{YY}
T.Yang, Y.Yang. {\it A Non-Iterative Reconstruction Algorithm for
the Acoustic Inverse Boundary Value Problem.} arXiv:2009.00641v1
[math.AP] 1 Sep 2020.

\end{thebibliography}
\end{document}